\documentclass{amsart}
\usepackage{url}
\usepackage{hyperref}
\newtheorem{lem}{Lemma}
\newcommand{\R}{{\mathbb{R}}}
\newcommand{\C}{{\mathbb{C}}}

\begin{document}
\title[Littlewood's proof of the FTA: a simpler version]{Littlewood's proof of the fundamental theorem of algebra: a simpler version}

\author{A. Bauval}

\begin{abstract} We present an elementary proof of the fundamental theorem of algebra, following Cauchy's version but avoiding his use of circular functions. It is written in the same spirit as Littlewood's proof of 1941, but reduces it to more elementary and constructive arguments.
\end{abstract}
\maketitle
%\renewcommand{\thefootnote}{}
%\footnote{2010 \emph{Mathematics Subject Classification}: 12D05, 30A10.}
%\footnote{\emph{Key words and phrases}: Fundamental theorem of algebra.}

\section{Introduction} Among the many ways to prove the FTA, which says that every nonconstant complex polynomial $P$ has at least one complex root, the most commonly taught since Cauchy \cite[pp. 331-339]{Cauchy} consists in showing that $|P|$ has a global minimum and that its value at any local minimum point has to be zero.

Neither of these two steps is completely elementary: the first one, after the straightforward remark that  $|P(z)|$ tends to infinity with $|z|$, uses -- in its modern formulation -- the Heine-Borel characterization of compact subspaces of the complex plane and the behaviour of continuous functions on such spaces; the second step {\sl seems to} involve the fact that any complex number has roots of arbitrary order $k$, which is obvious {\sl if} the use of circular functions is allowed.

In 1941, Littlewood \cite{Littlewood} recalled that this use of De Moivre's formula can be replaced by a less theoretically demanding algorithm \cite{Weber}, combining the case $k=2$ and the case where $k$ is odd: the former only relies on the existence of real square roots for nonnegative real numbers and the latter uses moreover the existence of a real zero for any real polynomial of odd degree. Since the latter property is, in turn, a corollary of the FTA, he found interesting to build two simple proofs of the FTA, relying only on the existence of square roots and on the minimum argument: one proof uses that argument again, and the other one replaces this repetition by an inductive reasoning.

More directly inspired by Cauchy's proof -- an analytic transcription of Argand's geometric one \cite{Argand} -- we shall get rid of both repetition and induction, and escape the use of square roots as well. Of course, there is no theoretical interest in this whole story, since the extreme value theorem is equivalent to the upper bound property \cite{Propp}, which characterizes $\R$ (among ordered fields), hence entails real-closedeness. Littlewood's sole aim was to give a cheaper version of the ``pedestrian'' proof of the FTA, and so is ours. Incidentally, we also replace the modern compactness arguments by the less sophisticated Bolzano theorem for real sequences.

Meanwhile writing this paper, we discovered that a similar one \cite{Oliveira} was published two years ago. We defer to it for more context and bibliography, and hope our definitely simpler version to show up worthwhile.

\section{Cauchy's proof revisited}

Let us first rephrase the three steps of Cauchy's proof in terms which avoid square roots -- hence only talk about $|z|^2=x^2+y^2$ instead of $|z|=\sqrt{x^2+y^2}$, unless $z$ is real. Using only Bolzano's theorem for real sequences -- which had just been proved in 1817 -- we shall also make explicit, as elementarily and faithfully as possible, Cauchy's informal justification of the existence of a minimum for the (square of the) absolute value of a complex polynomial of degree $n>0$,$$P(z)=a_0+a_1z+\ldots+a_nz^n.$$

\begin{lem}$\lim_{|z|^2\to+\infty}|P(z)|^2=+\infty.$
\end{lem}

\begin{proof}By the Cauchy-Schwarz inequality in $\C^{n+1}$,
$$|a_nz^n|^2=\left|P(z)-\sum_{k=0}^{n-1}a_kz^k\right|^2\le(n+1)\left(|P(z)|^2+\sum_{k=0}^{n-1}|a_kz^k|^2\right),$$
hence
$$|P(z)|^2\ge|z|^{2n}\left(\frac{|a_n|^2}{n+1}-\sum_{k=0}^{n-1}\frac{|a_k|^2}{|z|^{2(n-k)}}\right).$$
\end{proof}

\begin{lem}The function $|P|^2$ has a global minimum.
\end{lem}

\begin{proof}Let $m$ be the greatest lower bound of its values, and $(z_k)$ a sequence such that $|P(z_k)|^2\to m$. By the previous lemma, $(|z_k|^2)$ is bounded, hence so are the sequences of real and imaginary parts $x_k,y_k$ of $z_k$. By Bolzano's theorem,  $(x_k)$ has a convergence subsequence, $x_{\varphi(k)}\to x$, and $(y_{\varphi(k)})$, in turn, has a convergent subsequence, $y_{\psi(\varphi(k))}\to y$. By continuity of $|P|^2$ (viewed as a polynomial function of two real variables, with real coefficients), $|P(x+{\rm i}y)|^2=m$.
\end{proof}

\begin{lem}A point where $|P|^2$ does not vanish cannot be a local minimum.
\end{lem}

\begin{proof}For such a point $z$, there exists an integer $k>0$ and a complex polynomial $Q$ such that
$$P(z+h)=a+h^kQ(h),\quad a=P(z)\ne0,\quad b=Q(0)\ne0.$$
Fix some $h\in\C$ such that  ${\rm Re}(h^kb\overline a)<0$, and consider the real polynomial
$$|P(z+th)|^2=\left(a+t^kh^kQ(th)\right)\left(\overline a+t^k\overline{h^kQ(th)}\right)=b_0+b_kt^k+\ldots+b_{2n}t^{2n},$$
where $b_0=a\overline a$, and $b_k=2{\rm Re}(h^kb\overline a)<0$. Then, for any $t>0$ small enough (more precisely: $t\le1$ and $t<-b_k/\sum_{j>k}|b_j|$),
$$|P(z+th)|^2-|P(z)|^2=t^k\sum_{j\ge k}b_jt^{j-k}\le t^k\left(b_k+t\sum_{j>k}|b_j|\right)<0.$$
\end{proof}

When $k=1$, the existence of an $h$ like in the last proof is obvious. Cauchy used circular functions to justify the existence of such an $h$, even when $k>1$. The next section gives a more elementarily construction.

\section{Two constructive key lemmas}

The next lemma pretends to ignore the following inequality, but shows that it happens earlier than expected:
$$\lim_{k\to+\infty}{\rm Re}\left[(1+3{\rm i}/k)^k\right]=\cos 3<0.$$

\begin{lem}For any integer $k\ge2$, let $z_k=1+3{\rm i}/k.$ Then, ${\rm Re}(z_k^k)<0$.
\end{lem}

\begin{proof}The binomial formula gives a finite alternate sum
$${\rm Re}(z_k^k)=x_0-x_1+x_2-x_3+\ldots\quad{\rm where}\quad x_j={k\choose 2j}(3/k)^{2j},$$
whith the comfortable convention that ${k\choose \ell}=0$ for any $\ell>k$.\\
The first term is $x_0=1$ and the following ones can be grouped by pairs. For any integer $p\ge0$,
$$-x_{2p+1}+x_{2p+2}={k\choose 4p+2}\left(\frac3k\right)^{4p+2}\left(-1+\frac{9(k-4p-2)(k-4p-3)}{(4+4p)(3+4p)k^2}\right)\le0.$$
A more accurate upper bound for the first pair ends the proof:
$$-x_1+x_2<\frac{4(k-1)}k\left(-1+\frac{3(k-2)(k-3)}{4k^2}\right)=-1-\frac{14k^2-33k+18}{8k^3}<-1.$$

\end{proof}

This lemma gives the following constructive argument to conlude the proof of the FTA, as claimed at the end of the previous section:

\begin{lem}For any complex number $c\ne0$ and any integer $k\ge2$, there exist an $h\in\C$ such that ${\rm Re}(h^kc)<0$.
\end{lem}

\begin{proof}Let $\alpha,\beta$ be the real and imaginary parts of $c$. According to the property of $z_k$ in the previous lemma, in order to get $\alpha{\rm Re}(h^k)-\beta{\rm Im}(h^k)<0$, just choose:
\begin{itemize}
\item if $\alpha<0$: $h=1$,
\item if  $\alpha>0$: $h=z_k$ or its conjugate (the one for which $\beta{\rm Im}(h^k)\ge0$),
\item if $\alpha=0$ (hence $\beta\ne0$): $h=z_{2k}$ or its conjugate, chosen similarily (note that ${\rm Im}(z_{2k}^k)\ne0$, since $\left({\rm Re}(z_{2k}^k)\right)^2-\left({\rm Im}(z_{2k}^k)\right)^2={\rm Re}(z_{2k}^{2k})<0$).
\end{itemize}
\end{proof}

\author{Anne Bauval}\\
\address{\small Institut de Math\'ematiques de Toulouse\\
\' Equipe \' Emile Picard, UMR 5580\\
Universit\'e Toulouse III\\
118 Route de Narbonne, 31400 Toulouse - France\\
bauval@math.univ-toulouse.fr}

\end{document}